\documentclass[12pt]{article}

\usepackage{authblk}

\newcommand{\rmbf}[2][\mathbf]{\mathrm{#1{#2}}} 
\newcommand{\gbf}[1]{\boldsymbol{#1}}

\newcommand{\tr}{\mbox{tr}}
\newcommand{\rank}{\mathrm{rank}}
\newcommand{\p}[2][]{\ensuremath{_{(#2)#1}}}

\usepackage{amsthm,amssymb,amsmath}

\newtheoremstyle{mystyle}
{}
{}
{\itshape}
{}
{\bfseries}
{.}
{ }
{\thmname{#1}\thmnumber{ #2}\thmnote{ (#3)}}

\theoremstyle{mystyle}
\newtheorem{theorem}{Theorem}
\newtheorem{lemma}{Lemma}

\newtheoremstyle{myremark}
{}
{}
{}
{}
{\itshape}
{.}
{\newline}
{\thmname{#1}\thmnumber{ #2}\thmnote{ (#3)}}
\theoremstyle{remark} 
\newtheorem{remark}{Remark}

\newtheoremstyle{mycondition}
{}
{}
{}
{}
{\bfseries}
{}
{ }
{\thmname{#1}\thmnumber{ (#2)}\thmnote{ (#3)}}
\theoremstyle{mycondition} 
\newtheorem{condition}{Condition}

\usepackage{natbib,graphicx,setspace,subfigure,url,color,epstopdf,grffile,enumerate}
\usepackage[colorlinks,citecolor=blue,urlcolor=black,linkcolor=blue]{hyperref}

\def\spacingset#1{\renewcommand{\baselinestretch}%
	{#1}\small\normalsize} \spacingset{1}

\begin{document}

\title{Asymptotic properties of adaptive group Lasso for sparse reduced rank regression}

\author[1]{Kejun He\thanks{Email:~\href{mailto:kejun@stat.tamu.edu}{kejun@stat.tamu.edu}}}
\author[1]{Jianhua Z. Huang}
\affil[1]{Department of Statistics,\protect\\
	Texas A\&M University, College Station, TX 77843, USA}
\date{}

\maketitle

\begin{abstract}
This paper studies the asymptotic properties of the penalized least squares estimator using an adaptive group Lasso penalty for the reduced rank regression. The group Lasso penalty is defined in the way that the regression coefficients corresponding to each predictor are treated as one group. It is shown that under certain regularity conditions, the estimator can achieve the minimax optimal rate of convergence. Moreover, the variable selection consistency can also be achieved, that is, the relevant predictors can be identified with probability approaching one. In the asymptotic theory, the number of response variables, the number of predictors, and the rank number are allowed to grow to infinity with the sample size.
\end{abstract}

\noindent \textbf{Keywords:} {high dimensional regression; large sample theory; minimax; multivariate regression; oracle property; variable selection}

\newpage 

\spacingset{1.58} 

\section{The Model} \label{sec:model}
Suppose there are $q$ multiple response variables $Y_1,Y_2, \dots, Y_{q}$, and $p$ multiple predictors $X_1,X_2, \dots, X_{p}$.  The linear model assumes that 
\begin{equation*} 
Y_k=\sum\limits_{j=1}^p X_jc_{jk} + e_k, \quad k=1,2,\dots, q. 
\end{equation*}
Without loss of generality, we omit the intercept term in the linear model, since this term can be removed by assuming the response variables and the predictors have mean zero. We also assume that the $q$ error terms $e_k$, $k =1, \dots, q$, are random variables with mean zero. Suppose that we have an independent sample of size $n$ from this model. Let $\rmbf Y_k$, $1 \leq k \leq q$, and $\rmbf X_j$, $1 \leq j \leq p$ denote the $n$-dimensional response vector and predictor vector respectively. Let $\rmbf{Y}=(\rmbf{Y}_1,\rmbf{Y}_2, \dots, \rmbf{Y}_{q})$ and $\rmbf{X}=(\rmbf{X}_1,\rmbf{X}_2, \dots, \rmbf{X}_{p})$ be the $n \times q$ and $n \times p$ data matrices respectively. The model for the observed data can be written as
\begin{equation} \label{eqn:model}
\rmbf Y=\rmbf X \rmbf C + \rmbf E, 
\end{equation}
where $\rmbf C$ is the $p \times q$ matrix of regression coefficients and $\rmbf E$ is the $n \times q$ error matrix. Reduced rank regression \citep{Izenman1975} is an effective way of taking into account the possible interrelationships between the response variables by imposing a constraint on the rank of $\rmbf{C}$ to be less than or equal to $r, \, r \leq \min(p,q)$. We can estimate the rank-constraint coefficient matrix by solving the optimization problem
\begin{equation*}
\min_{\textbf C: \rank(\textbf C) \leq r} ||\rmbf Y - \rmbf X \rmbf C||^2.
\end{equation*}

Let $\rmbf C_j^T$ denote the $j$-th row of $\rmbf C$, which is the coefficient vector corresponding to $X_j$. Note that $\rmbf C_j$ being a zero vector indicates that the $X_j$ is irrelevant in predicting the responses. An estimation method that can simultaneously select relevant predictors has the property that it may produce some zero coefficient vectors. \cite{chen2012sparse} considered solving the following optimization problem for variable selection and estimation: 
\begin{equation} \label{eqn:chen optimization}
\min_{\textbf{C}: \rank(\textbf C) \leq r} ||\rmbf Y - \rmbf X \rmbf C||^2 +n \sum\limits_{j=1}^p \lambda_j||\rmbf C_j||, 
\end{equation}
where $||\cdot||$ represents the Frobenius norm and $\lambda_j \ge 0$ is a penalty parameter, $1 \leq j \leq p$. The resulting estimator worked well in simulation studies and real data applications. To study the asymptotic properties, \cite{chen2012sparse} only considered the asymptotic behavior of one ``local minimum'' of \eqref{eqn:chen optimization} when $n \to \infty$ and $p\, , q\, ,r$ are fixed constants. The purpose of this paper is to further the asymptotic theory of \cite{chen2012sparse} in two aspects:
(i) we consider the asymptotic behavior of the global minimum of \eqref{eqn:chen optimization};
(ii) we allow $p$, $q$, $r$ tend to infinity with the sample size $n$ and, in particular, the number of predictors can tend to infinity at a rate faster than that of the number of observations.

If we remove the rank constraint, vectorize the coefficient matrix $\rmbf C$ by rows, and suppose $\lambda_j=\lambda$, $1 \leq j \leq p$, then the penalty term in \eqref{eqn:chen optimization} is the same as the group Lasso penalty \citep{Yuan2006}. The asymptotic behavior of linear regression with the group Lasso penalty and its variations, like the adaptive group Lasso penalty, have been discussed by \cite{lounici2011oracle,wei2010consistent}. However, the techniques and results in these articles cannot be directly applied in reduced rank regression, since the low rank assumption introduces a manifold structure that will make the domain of coefficient matrix non-convex and invalidate the application of the Karush–Kuhn–Tucker condition, an essential tool used in these work.

In a related paper with a broader scope of that also discussed rank selection, \cite{bunea12} studied the rate of convergence for prediction based on the estimator that solves the penalized least squares problem \eqref{eqn:chen optimization} with all $\lambda_j$ being equal; they showed that the estimator achieves, with a logarithmic factor $\log(p)$, the optimal rate of convergence for prediction. We improves this earlier result by showing that when using an adaptive Lasso penalty, our estimator can achieve the optimal rate of convergence for prediction, without the extra logarithmic factor. We are able to show that the variable selection is consistent and that the convergence rate of the coefficient matrix $\rmbf C$ can be the same as when we know a priori which predictors are relevant. These issues were not considered in the previous work.

\section{Main Results} \label{sec:mainResults}
Let $s$ denote the number of relevant predictors, i.e., the number of $\rmbf C_j$'s that are not zero vectors. We allow $p, \, q, \, s,$ and $r$ ($r \leq \min(p,q)$) to grow with $n$. Without loss of generality, let the first $s$ predictors be the relevant ones. Let $\rmbf X\p{1}$, $\rmbf X\p2$ contain the columns of $\rmbf X$, and $\rmbf C\p{1}$, $\rmbf C\p2$ contain the rows of $\rmbf C$ associated with relevant predictors and irrelevant predictors, respectively. 
We also denote $\rmbf \Sigma$ as the Gram matrix of $\rmbf X/\sqrt{n}$, i.e. 
	$$\rmbf \Sigma = \frac{\rmbf X^T \rmbf X}{n}.$$
Analogously, $\rmbf \Sigma\p1$, $\rmbf \Sigma\p2$ are Gram matrices for $\rmbf X\p1/\sqrt{n}$ and $\rmbf X\p2/\sqrt{n}$ respectively. We use $C$ to refer to a generic constant that may change values from context to context, and let $a \ll b$ and $a \lesssim b$ mean $a = o(b)$ and $a = O(b)$ respectively. When $p > n$, $\rmbf \Sigma$ is a degenerate matrix, in the sense that the smallest eigenvalue of $\rmbf \Sigma$ is $0$. Clearly, the ordinary least squares estimator does not work in this case.

We need the following regularity conditions on the design matrix, the error matrix, and the tuning parameters.
\begin{condition} \label{con:restricted eigenvalue}
	There exists a positive constant $C$ such that
	\begin{equation} \label{eqn:restriced eigenvalue}
	  \tr (\rmbf M^T \rmbf \Sigma \rmbf M) \ge C \sum\limits_{j \in \{1, \dots, s\}}||\rmbf m_j||^2
	\end{equation}
	for all $p \times q$ matrices $\rmbf M$ (with rows $\rmbf m_j$) whenever $\sum_{j \in \{s+1, \dots,p\}} ||\rmbf m_j|| \leq 2 \sum_{j \in \{1, \dots, s\}} ||\rmbf m_j||$. 
\end{condition}	

\begin{condition} \label{con:greatest eigenvalue}
	The greatest eigenvalue of $\rmbf \Sigma$ is bounded away from $\infty$.
\end{condition}

\begin{condition} \label{con:error sub gaussian}
	The noise matrix $\rmbf E = (\rmbf E_1, \dots, \rmbf E_n)^T$ has independent and identically distributed rows, with the vector $\rmbf E_i$ being sub-Gaussian in the sense that $\mathbb{E}\exp(t\mathbf{E}_i^T \gbf \eta) \leq \exp(C t^2 ||\gbf \eta||^2)$ for any $\gbf \eta \in \mathbb R^q$. 
\end{condition}

\begin{condition} \label{con:variate lasso lambda}
	 For $\lambda\p1=\max_{1 \leq j \leq s} \lambda_j$ and $\lambda\p2 = \min_{(s+1) \leq j \leq p} \lambda_j$, we have that $\lambda\p1 \lesssim \sqrt{{r(q+s-r)}/{ns}}$ and $\lambda\p2 \gg \sqrt{ \{q \log(p)+r (q+s-r) \} /{n}}$.
\end{condition}

Condition \eqref{con:restricted eigenvalue} is similar to but slightly different from the ``restricted eigenvalue'' (RE) condition introduced by \cite{bickel09} for studying the asymptotic properties of Lasso regression. This condition implies that (i) the number of relevant predictors is less than the number of observations and, (ii) the least eigenvalue of $\rmbf \Sigma\p1$ is greater than or equal to $C$ by letting $\rmbf m_j= \rmbf 0$, $s+1 \leq j \leq p$, which is a necessary condition to identify $\rmbf C\p1$ if we have known which variables are relevant. From the proofs of the theorems, it can be shown that the constant $2$ of the assumption $\sum_{j \in \{s+1, \dots,p\}} ||\rmbf m_j|| \leq 2 \sum_{j \in \{1, \dots, s\}} ||\rmbf m_j||$ in Condition \eqref{con:restricted eigenvalue} can be replaced by any constant $C > 1$. 

Condition \eqref{con:greatest eigenvalue} and \eqref{con:error sub gaussian} are commonly seen in the regression literature \citep{chen2012sparse,Izenman2008,wei2010consistent}. In particular,  Condition \eqref{con:greatest eigenvalue} implies that all the diagonal elements in $\rmbf \Sigma$ are bounded away from $\infty$ and also the greatest singular value of $(\rmbf X\p1^T \rmbf X\p2/n)$ is bounded away from $\infty$. Condition \eqref{con:error sub gaussian} is used to control the stochastic error. For Condition \eqref{con:variate lasso lambda}, the upper bound of $\lambda\p1$ is used to identify the relevant predictors, and the lower bound of $\lambda\p2$ is used to annihilate the irrelevant predictors.

Let $\rmbf C^0$ be the true value of the coefficient matrix of $\rmbf C$. 
\begin{theorem}[Oracle Properties of the Estimator] \label{thm:convergence rate}
Assume Conditions \eqref{con:restricted eigenvalue} -- \eqref{con:variate lasso lambda} are satisfied. Then the solution $\widehat{\rmbf C}$ of \eqref{eqn:chen optimization} has the following properties:
	\begin{enumerate}
		\item the $L_2$ norm of prediction error $\|\rmbf X\widehat{\rmbf C} - \rmbf X \rmbf C^0\|$ is upper bounded by $O_p(\sqrt{r(q+s-r)})$; 
		\item $\mathbb P(\widehat{\rmbf C}_j = \rmbf 0, \, s+1 \leq j \leq p) \to 1$ as $n \to \infty$;
		\item  $\widehat{\rmbf C}$ is converging to $\rmbf C^0$ with the rate upper bounded by $O_p(\sqrt{{r(q+s-r)}/n})$.
	\end{enumerate}
\end{theorem}

\begin{remark}
  With a straightforward modification of Theorem 5 in \cite{koltchinskii2011nuclear}, the minimax lower bound of $L_2$ norm of prediction error for this model is $O_p(\sqrt{r(q\vee s)}) \asymp O_p(\sqrt{r(q+s-r)})$, where $q \vee s = \max\{q,s\}$. The first result of Theorem \ref{thm:convergence rate} says that our estimator achieves the minimax optimal rate of convergence for prediction. The $L_2$ norm of prediction error obtained in \cite{bunea12} is $O_p(\sqrt{r\{q + s \log(p)\} })$, which has an extra logarithmic factor.
\end{remark}

\begin{remark}
	The second result of Theorem \ref{thm:convergence rate} says that the penalized estimator can consistently identify the relevant predictors. The proof of this theorem also shows that the estimator with the knowledge of which predictors are relevant has the same convergence rate given in the third result.
\end{remark}

 One drawback of Theorem \ref{thm:convergence rate} is that the $p$ penalty parameters $\lambda_j$'s have to be specified to satisfy Condition \eqref{con:variate lasso lambda}. Using the idea of adaptive (group) Lasso \citep{zou2006adaptive,wei2010consistent}, we now specify these $p$ penalty parameters using a single penalty parameter multiplied by a power of a certain pilot estimator. More precisely, we first solve problem \eqref{eqn:chen optimization} with a single penalty parameter to get a pilot estimator of $\rmbf C$, i.e., denote that 
 \begin{equation} \label{eqn:lasso estimator}
   \widehat{\rmbf C}^{\mathrm{Lasso}} = \min_{\textbf{C}: \rank(\textbf C) \leq r} ||\rmbf Y - \rmbf X \rmbf C||^2 +n \sum\limits_{j=1}^p \lambda^{\mathrm{Lasso}} \,||\rmbf C_j||.
 \end{equation}
 Then let the $p$ penalty parameters be 
 \begin{equation}\label{eqn:lambdas in adaptive lasso} 
 \lambda_j = \left\{\begin{aligned} 
 & \lambda^{\mathrm{Adap}} \, ||\widehat{\rmbf C}^{\mathrm{Lasso}}_j||^{-\beta}, & \mbox{ if} \quad \widehat{\rmbf C}^{\mathrm{Lasso}}_j \not= \rmbf 0, \\
 & +\infty, & \mbox{ if} \quad \widehat{\rmbf C}^{\mathrm{Lasso}}_j = \rmbf 0, 
 \end{aligned}\right.
 \end{equation}
 with a penalty parameter $\lambda^{\mathrm{Adap}}$ and some fixed $\beta >0$. Denote $\widehat{\rmbf C}^{\mathrm{Adap}}$ as the penalized estimator that is the solution of \eqref{eqn:chen optimization} with $\lambda_j$ defined in \eqref{eqn:lambdas in adaptive lasso}, $1 \leq j \leq p$. 
 
We now show that the penalty parameters as defined in \eqref{eqn:lambdas in adaptive lasso} satisfy Condition \eqref{con:variate lasso lambda} and so the resulting estimator $\widehat{\rmbf C}^{\mathrm{Adap}}$ enjoys the nice asymptotic properties as stated in Theorem \ref{thm:convergence rate}. We need the following regularity conditions.
 \begin{condition} \label{con:lasso lambda}
 	  $\sqrt{{q \log(p) }/{n}} \ll  \lambda^{\mathrm{Lasso}} \lesssim {n^{-\epsilon}}/{s} $ for some $\epsilon > 0$.
 \end{condition}
 
 \begin{condition}\label{con:uniform lower bound for all C}
 	$\|\rmbf C_j^0\|$ has a positive lower bound, i.e., $\min_{j \in \{1,\dots,s \}}\|\rmbf C_j^0\| \geq C > 0$ for some $C$.
 \end{condition}
 
 \begin{condition}\label{con:condition on beta}
 	For $\lambda^{\mathrm{Adap}}$, we have that $\sqrt{ \{q \log(p)+r (q+s-r) \} /n^{1+2\,\epsilon\beta}} \ll \lambda^{\mathrm{Adap}} \lesssim \sqrt{r(q+s-r)/ns}$.
 \end{condition}

Condition \eqref{con:lasso lambda} is similar to one of the regularity conditions in \cite{wei2010consistent}. It implies that $p = o(\exp(n))$, i.e., the number of predictors cannot tend to infinity faster than the exponential of the sample size. Conditions \eqref{con:lasso lambda} and \eqref{con:uniform lower bound for all C} guarantee the pilot estimator defined in \eqref{eqn:lasso estimator} has a certain convergence rate and the relevant predictors can be identified. Similar regularity condition on $\lambda^{\mathrm{Adap}}$ in Condition \eqref{con:condition on beta} was first seen in \cite{zou2006adaptive}. As a special case, if we restrict $\beta =1$, then Condition \eqref{con:condition on beta} can be replaced by order requirements on $n$, $r$, $p$ and $q$, similar to Condition (C3)* in \cite{wei2010consistent}. In particular, the first, second and forth requirement of Condition (C3)* in \cite{wei2010consistent} implied the lower bound of $\lambda^{\mathrm{Adap}}$, while its third requirement indicated the upper bound.
 
\begin{lemma} \label{lem:group Lasso}
	Assume Conditions \eqref{con:restricted eigenvalue} -- \eqref{con:error sub gaussian} and \eqref{con:lasso lambda} hold. Then the solution $\widehat{\rmbf C}^{\mathrm{Lasso}}$ of \eqref{eqn:lasso estimator} satisfies $||\widehat{\rmbf C}^{\mathrm{Lasso}} - \rmbf C^0|| \leq O_p(n^{-\epsilon})$. If we further assume that Condition \eqref{con:uniform lower bound for all C} is satisfied, then $\|\widehat{\rmbf C}^{\mathrm{Lasso}}_j \| \not= 0$, $j= 1, \dots, s$, with probability approaching 1. 
\end{lemma}

\begin{theorem} \label{thm:adaptive group Lasso}
	Assume Conditions \eqref{con:restricted eigenvalue} -- \eqref{con:error sub gaussian} and \eqref{con:lasso lambda} -- \eqref{con:condition on beta} hold. Then the estimator $\widehat{\rmbf C}^{\mathrm{Adap}}$ has the properties stated in Theorem \ref{thm:convergence rate}. Moreover, the relevant predictors can also be identified with probability approaching $1$, i.e., $\mathbb P(\widehat{\rmbf C}^{\mathrm{Adap}}_j \neq \rmbf 0, \, 1 \leq j \leq s) \to 1$.
\end{theorem}

\section{Preliminary Lemmas} \label{sec:preliminaryLemmas}
The following two lemmas play important roles in the proofs of the main results.

\begin{lemma} \label{lem:oracle estimator}
Suppose that the eigenvalues of $\rmbf \Sigma\p1$ are bounded away from $0$ and $\infty$ and that the $n \times q$ error matrix $\rmbf E$ satisfies Condition \eqref{con:error sub gaussian}. Then 
\begin{equation} \label{eqn:error order}
	\langle \rmbf X\p1 (\rmbf C\p1^1-\rmbf C\p1^2), \rmbf E \rangle \leq \|{\rmbf C}\p1^1-\rmbf C^2\p1 \| \cdot O_p(\sqrt{nr(q+s-r)}),
\end{equation}
for any $s \times q$ matrices $\rmbf C^1\p1$ and $\rmbf C\p1^2$ with rank less than or equal to $r$.
\end{lemma}

\begin{proof}[Proof of Lemma \ref{lem:oracle estimator}]
	Let $\Gamma=\{\gbf \eta=\rmbf X\p1 \rmbf C\p1 / \sqrt{n \, \lambda_{\max}(\rmbf \Sigma\p1)}: \|\rmbf C\p1 \|\le 1, {\rm rank}( \rmbf C\p1)\le r \}$, where $\lambda_{\max}(\cdot)$ denotes the largest eigenvalue of a symmetric matrix. 
	We first show that the covering entropy $\log N(\epsilon,\Gamma,l_2)\le r(q+s-r)\log(C /\epsilon)$, where $l_2$ denotes the Frobenius norm. In fact, for $\gbf \eta=\rmbf X\p1 \rmbf C\p1 /\sqrt{n \, \lambda_{\max}(\rmbf \Sigma\p1)} \in \Gamma$ with $\|\rmbf C\p1 \|\le 1, {\rm rank}( \rmbf C\p1)\le r$, we can write $\rmbf C\p1:=\rmbf D \rmbf A^T$, $\rmbf D\in \mathbb{R}^{s\times r}$, $\rmbf A\in \mathbb{R}^{q \times r}$, $\|\rmbf D\|\le 1$, $\rmbf A^T \rmbf A=\rmbf I_r$. Let $\mathcal D = \{\rmbf D: \rmbf D \in \mathbb R^{s \times r}, \, \|\rmbf D \| \leq 1\}$ and $\mathcal A = \{ \rmbf A: \rmbf A \in \mathbb R^{q \times r}, \,  \rmbf A^T \rmbf A = \rmbf I_r \}$. In the following, also let $\rmbf D_1$, $\rmbf D_2 \in \mathcal D$ and $\rmbf A_1$, $\rmbf A_2 \in \mathcal A$ respectively.
	The covering number of $\mathcal D$, under the Frobenius norm, is bounded by $(C/\epsilon)^{rs}$. According to Proposition 8 of \cite{szarek82}, the covering number of $\mathcal A$ is bounded by $(C/\epsilon)^{r(q-r)}$ under the distance defined as $d(\rmbf A_1,\rmbf A_2)=\|\rmbf A_1\rmbf A_1^T-\rmbf A_2\rmbf A_2^T\|_{op}$, where $\|.\|_{op}$ denotes the operator norm. To obtain the covering number of $\Gamma$, note that 
	\begin{align*}
		\|\rmbf D_1\rmbf A_1^T-\rmbf D_2\rmbf A_2^T\| & \le  \|\rmbf D_1\rmbf A_1^T-\rmbf D_1\rmbf A_1^T \rmbf A_2\rmbf A_2^T\| + \|\rmbf D_1\rmbf A_1^T \rmbf A_2\rmbf A_2^T-\rmbf D_2\rmbf A_2^T\|\\
		& \le  \|\rmbf D_1 \rmbf A_1^T \| \|\rmbf A_1\rmbf A_1^T- \rmbf A_2 \rmbf A_2^T\|_{op}+\|\rmbf D_1\rmbf A_1^T\rmbf A_2-\rmbf D_2\| \\
		& \le \|\rmbf A_1\rmbf A_1^T- \rmbf A_2 \rmbf A_2^T\|_{op}+\|\rmbf D_1\rmbf A_1^T\rmbf A_2-\rmbf D_2\|,
	\end{align*}
	where the third inequality is because $\|\rmbf D_1 \| \leq 1$ and $\rmbf A_1^T \rmbf A_1 = \rmbf I$.  Since $\rmbf D_1 \rmbf A_1^T \rmbf A_2 \in \mathbb{R}^{s \times r }$ and $\| \rmbf D_1 \rmbf A_1^T \rmbf A_2 \| \leq 1$, we have $\rmbf D_1 \rmbf A_1^T \rmbf A_2 \in \mathcal D$.  Using the covering numbers of $\mathcal D$ and $\mathcal A$, it is shown that $N(\epsilon,\Gamma,l_2)\le (C/\epsilon)^{rs} \cdot (C/\epsilon)^{r(q-r)}$, and thus $\log N(\epsilon,\Gamma,l_2)\le r(q+s-r)\log(C /\epsilon)$.
		
	Furthermore, it follows from the sub-Gaussian error assumption of Condition \eqref{con:error sub gaussian} that
	\begin{equation*}
	\mathbb E\exp( t\langle \rmbf E, \gbf \eta\rangle) \le  \exp(Ct^2\|\gbf \eta\|^2).
	\end{equation*}
	Using Dudley's integral entropy bound \citep[for example, see Theorem 3.1 of][]{koltchinskii11}, we get
	$$\mathbb{E} \sup_{\gbf \eta \in\Gamma}\langle \gbf \eta,\rmbf E \rangle \le C\int_0^2 \sqrt{ r(q+s-r)\log (\frac{C}{\epsilon})}\,d\epsilon \le C \sqrt{ r(q+s-r)}.$$
	The above implies that
	$$\langle \rmbf X\p1 ({\rmbf C}^1\p1-\rmbf C^2\p1)/ \sqrt{n\, \lambda_{\max}(\rmbf \Sigma\p1)}, \rmbf E \rangle \leq \| {\rmbf C^1}\p1-\rmbf C^2\p1 \| \cdot O_p(\sqrt{r(q+s-r)}),$$
	which in turn gives \eqref{eqn:error order} by the assumption on $\rmbf \Sigma\p1$.
\end{proof}

\begin{lemma} \label{lem:error order2}
	Assume Conditions \eqref{con:greatest eigenvalue} -- \eqref{con:error sub gaussian} are satisfied. Then
	\begin{equation} \label{eqn:error order2}
	    |\langle \rmbf E, \rmbf X \rmbf C \rangle| \leq O_p(\sqrt{n q \log(p)}) \sum_{j =1}^p \|\rmbf C_j\|,
	\end{equation}
	for any $p \times q$ random matrix $\rmbf C$, where $\rmbf C_j^T$ is the $j$-th row of $\rmbf C$.
\end{lemma}

\begin{proof}[Proof of Lemma \ref{lem:error order2}]
	To show this result, we first note that
	\begin{equation} \label{eqn:from C to Cj}
	|\langle \rmbf E, \rmbf X \rmbf C \rangle| = |\langle \rmbf X^T \rmbf E , \rmbf C \rangle | \leq \xi \sum\limits_{j=1}^p ||\rmbf C_j||,
	\end{equation}
	where $\xi = \max_{j} \|\rmbf X_j^T \rmbf E \|$. Secondly, Condition \eqref{con:greatest eigenvalue} and \eqref{con:error sub gaussian} imply that 
	\begin{equation} \label{eqn:XENorm}
	\|\rmbf X_j^T \rmbf E  \|  \leq O_p(\sqrt{nq})
	\end{equation}
	 Finally, a direct application of Lemma 2.2.2 of \cite{vaartwellner96} shows that $\xi$ is upper bounded by 
	\begin{equation} \label{eqn:xi sub 2}
		 \xi\leq O_p(\sqrt{nq \log(p)}). 
	\end{equation} 
	Thus, \eqref{eqn:from C to Cj}, \eqref{eqn:XENorm} and \eqref{eqn:xi sub 2} complete the proof.
\end{proof}

\section{Proof of Main Results} \label{sec:proofs} 

\begin{proof}[Proof of Theorem \ref{thm:convergence rate}]
	We organize the proof in the following order: we first show that $\|\rmbf X\widehat{\rmbf C} - \rmbf X \rmbf C^0\| \leq O_p(\sqrt{r(q+s-r)})$ and $\|\widehat{\rmbf C}\p1 - \rmbf C\p1^0 \| \leq O_p(\sqrt{r(q+s-r)/n})$; we then show $\mathbb P(\widehat{\rmbf C}\p2= \rmbf 0) \to 1$ as $n \to \infty$; the third property in the theorem is an immediate result of the above two by noting that $\|\widehat{\rmbf C} - \rmbf C^0\|^2= \|\widehat{\rmbf C}\p1 - \rmbf C\p1^0 \|^2 + \|\widehat{\rmbf C}\p2\|^2$. 
	
	Let $$Q(\rmbf C)= ||\rmbf Y - \rmbf X \rmbf C||^2 +n \sum\limits_{j=1}^p \lambda_j||\rmbf C_j||.$$
	$\widehat{\rmbf C}$ is the solution indicating that
	\begin{equation} \label{eqn:Qc less than Qc0}
       Q(\widehat{\rmbf C}) \leq Q(\rmbf C^0),
	\end{equation}
	where $\rmbf C^0$ is the true parameter matrix with $\rank(\rmbf C^0) \leq r$.
    Denote $\sum_{j =1}^s||\widehat{\rmbf C}_j - \rmbf C^0_j ||=\delta_1$ and $\sum_{j =s+1}^p||\widehat{\rmbf C}_j||=\delta_2$. Note that we have $\rank(\widehat{\rmbf C}\p1) \leq \rank(\widehat{\rmbf C}) \leq r$, and 
	\begin{align} \label{eqn:difference between Q}
	   & \;Q(\widehat{\rmbf C}) - Q(\rmbf C^0) \nonumber \\
	   = & \; ||\rmbf Y - \rmbf X \widehat{\rmbf C}||^2 +n \sum\limits_{j=1}^p \lambda_j||\widehat{\rmbf C}_j|| - ||\rmbf Y - \rmbf X \rmbf C^0||^2
	   - n \sum\limits_{j=1}^p \lambda_j||\rmbf C^0_j||  \nonumber \\
	  = &\; ||\rmbf X(\widehat{\rmbf C}- \rmbf C^0)||^2 - 2 \langle \rmbf{Y-XC^0},\rmbf X(\widehat{\rmbf C} - \rmbf C^0) \rangle + n \sum\limits_{j=1}^{p} \lambda_j \, ||\widehat{\rmbf C}_j|| - n \sum\limits_{j=1}^s \lambda_j ||\rmbf C^0_j|| \nonumber \\ 
	  =& \; ||\rmbf X(\widehat{\rmbf C}- \rmbf C^0)||^2 - 2 \langle \rmbf{E},\rmbf X(\widehat{\rmbf C} - \rmbf C^0) \rangle + n \sum\limits_{j=1}^{p} \lambda_j \, ||\widehat{\rmbf C}_j|| - n \sum\limits_{j=1}^s \lambda_j ||\rmbf C^0_j||.
	\end{align}
	\eqref{eqn:Qc less than Qc0} and \eqref{eqn:difference between Q} imply that 
	\begin{align} \label{eqn:two point inequality}
	   & \; ||\rmbf X(\widehat{\rmbf C}- \rmbf C^0)||^2 \nonumber \\
	   \leq & \;  2 \langle \rmbf{E},\rmbf X(\widehat{\rmbf C} - \rmbf C^0) \rangle + n \sum\limits_{j=1}^s \lambda_j ||\rmbf C^0_j|| - n \sum\limits_{j=1}^{p} \lambda_j \, ||\widehat{\rmbf C}_j||  \nonumber \\
	   \leq & \;  2 \langle \rmbf{E},\rmbf X(\widehat{\rmbf C} - \rmbf C^0) \rangle + n \sum\limits_{j=1}^s \lambda_j ||\widehat{\rmbf C}_j - \rmbf C^0_j|| -  n \sum\limits_{j=s+1}^{p} \lambda_j  ||\widehat{\rmbf C}_j || \nonumber \\
	   \leq & \;   2 \langle \rmbf{E},\rmbf X(\widehat{\rmbf C} - \rmbf C^0) \rangle + n \, \lambda\p1 \delta_1 -  n \, \lambda\p2 \delta_2,
	\end{align}
	where the second line is because of the triangular inequality, and $\lambda\p1$, $\lambda\p2$ are defined as in Condition \eqref{con:variate lasso lambda}.
	
	We now consider two cases: $\delta_2 > 2 \, \delta_1$ and $\delta_2 \leq 2 \, \delta_1$. We will show that $\mathbb P(\delta_2 > 2 \, \delta_1)$ is approaching $0$ and on the event of $\delta_2 \leq 2 \, \delta_1$, the rate of $\|\rmbf X(\widehat{\rmbf C} - \rmbf C^0) \|$ is upper bounded by $O_p(\sqrt{r(q +s -r)})$.
	
	To show $\mathbb P(\delta_2 > 2 \, \delta_1) \to 0$, we first note that
	\newpage
	\begin{align} \label{eqn:case 2 main}
	 2\, |\langle  \rmbf E,\rmbf X(\widehat{\rmbf C} - \rmbf C^0) \rangle | 
	 &=  2\, |\langle  \rmbf E,\rmbf X\p1(\widehat{\rmbf C}\p1 - \rmbf C^0\p1)+ \langle \rmbf E, \rmbf X\p2 \widehat{\rmbf C}\p2 \rangle | \nonumber \\
	 & \leq  2\, |\langle  \rmbf E,\rmbf X\p1(\widehat{\rmbf C}\p1 - \rmbf C^0\p1)\rangle|+2\, | \langle \rmbf E, \rmbf X\p2 \widehat{\rmbf C}\p2 \rangle | \nonumber \\
	&\leq  O_p (\sqrt{n q \log(s)} ) \, \delta_1 + O_p(\sqrt{nq \log(p)}) \, \delta_2,
	\end{align}
	where the first part of the third line is by Lemma \ref{lem:error order2} with the number of predictor variables equal to $s$, and the second part of the third line is a direct application of Lemma \ref{lem:error order2}.
	
	Therefore, given the event of $\delta_2 > 2 \, \delta_1$, \eqref{eqn:two point inequality} and \eqref{eqn:case 2 main} imply a contradiction such that
	\begin{align} \label{eqn:difference of Q1 case 2}
	 0
	& \leq   ||\rmbf X(\widehat{\rmbf C}- \rmbf C^0)||^2 \nonumber \\
	&  \leq   \{ O_p(\sqrt{n q \log(p)}) - n \, \lambda\p2 \} \delta_2  
	+ \{ O_p (\sqrt{n q \log(s)}) + n \,  \lambda\p1  \} \delta_1 \nonumber \\
	& <   \{ O_p(\sqrt{n q  \log(p)})  -2 n \, \lambda\p2
	+ O_p (\sqrt{n q \log(s)}) + n \,  \lambda\p1  \} \delta_1 \nonumber \\
	& <  0
	\end{align}
	with conditional probability approaching $1$, where the third and forth inequalities in \eqref{eqn:difference of Q1 case 2} are because of the second part of Condition \eqref{con:variate lasso lambda}. This contradiction implies that $\mathbb P(\delta_2 < 2 \delta_1)$ is approaching $0$. 
	
	Given the event of $\delta_2 \leq 2 \delta_1$, by Condition \eqref{con:restricted eigenvalue}, we have 
	\begin{equation}  \label{eqn:case 1 using lemma 1}
	||\rmbf X(\widehat{\rmbf C}- \rmbf C^0)||^2 = n\, \mathrm{tr} \{(\widehat{\rmbf C} - \rmbf C^0)^T \rmbf \Sigma (\widehat{\rmbf C} - \rmbf C^0) \} \geq n \, C \|\widehat{\rmbf C}\p1 - \rmbf C^0\p1 \|^2.
	\end{equation}
   It is shown by Lemma \ref{lem:oracle estimator} and Lemma \ref{lem:error order2} that
   \allowdisplaybreaks
	\begin{align} \label{eqn:case 1 main}
	   &\; 2\, |\langle  \rmbf E,\rmbf X(\widehat{\rmbf C} - \rmbf C^0)  \rangle | \nonumber \\
	   = & \;  2\, |\langle  \rmbf E,\rmbf X\p1(\widehat{\rmbf C}\p1 - \rmbf C^0\p1)+ \langle \rmbf E, \rmbf X\p2 \widehat{\rmbf C}\p2 \rangle | \nonumber \\
	   \leq & \;  2\, |\langle  \rmbf E,\rmbf X\p1(\widehat{\rmbf C}\p1 - \rmbf C^0\p1)\rangle|+2\, | \langle \rmbf E, \rmbf X\p2 \widehat{\rmbf C}\p2 \rangle | \nonumber \\
	   \leq & \;  O_p (\sqrt{nr(q+s-r)} ) \|\widehat{\rmbf C}\p1 -\rmbf C^0\p1 \| + O_p(\sqrt{nq \log(p)}) \sum\limits_{j=s+1}^p  ||\widehat{\rmbf C}_j||
	\end{align}
	Therefore, \eqref{eqn:two point inequality}, \eqref{eqn:case 1 using lemma 1}, \eqref{eqn:case 1 main}, and the assumptions of this theorem imply that
	\begin{align} \label{eqn:difference of Q1 case 1}
	   & \; n  \, C \, \|\widehat{\rmbf C}\p1 - \rmbf C^0\p1 \|^2  \nonumber \\
	   \leq & \;   ||\rmbf X(\widehat{\rmbf C}- \rmbf C^0)||^2 \nonumber \\
	   \leq & \;   \{O_p(\sqrt{n q  \log(p)})- n \, \lambda\p2 \} \delta_2 + O_p (\sqrt{nr(q+s-r)} )  \|\widehat{\rmbf C}\p1 - \rmbf C^0\p1 \| + n \,  \lambda\p1 \delta_1 \nonumber \\
	   \leq & \;    \{O_p(\sqrt{n q  \log(p)})- n \, \lambda\p2 \} \delta_2 + \{ O_p (\sqrt{nr(q+s-r)} )  + n\sqrt{s} \,  \lambda\p1  \} \|\widehat{\rmbf C}\p1 - \rmbf C^0\p1 \| \nonumber \\
	  \leq & \;   \{ O_p (\sqrt{nr(q+s-r)} )  + n\sqrt s \,  \lambda\p1  \} \|\widehat{\rmbf C}\p1 - \rmbf C^0\p1 \|, 
	\end{align}
	where the third inequality in \eqref{eqn:difference of Q1 case 1} is because 
	\begin{equation} \label{eqn:CauchyOnDelta1}
	  \delta_1 = \sum_{j =1}^s \|\widehat{\rmbf C}_j - \rmbf C^0_j \| \leq \sqrt{s} \|\widehat{\rmbf C}\p1 - \rmbf C^0\p1 \|
	\end{equation}
	by the Cauchy-Schwarz inequality. Thus it can be shown by the first part of Condition \eqref{con:variate lasso lambda} that 
	\begin{equation} \label{eqn:C1 convergence rate}
	\|\widehat{\rmbf C}\p1 - \rmbf C^0\p1 \| \leq O_p \left(\sqrt{\frac{r(q+s-r)}{n}} \right).
	\end{equation}
	Plugging this result into \eqref{eqn:difference of Q1 case 1} implies that $$||\rmbf X(\widehat{\rmbf C}- \rmbf C^0)||^2 \leq O_p(r(q+s-r)),$$ which is the first property in Theorem \ref{thm:convergence rate}.
	
	To show $\mathbb P(\widehat{\rmbf C}\p2= \rmbf 0) \to 1$, denote $\widetilde{\rmbf C}=(\widehat{\rmbf C}\p1^T, \rmbf 0^T )^T$ as the thresholding estimator of $\widehat{\rmbf C}$. It is obvious that $\rank(\widetilde{\rmbf C}) \leq \rank(\widehat{\rmbf C}) \leq r$, and $||\widetilde{\rmbf C}-\rmbf C^0||  = ||\widehat{\rmbf C}\p1-\rmbf C^0\p1||$. Moreover,
	$$Q(\widehat{\rmbf C})- Q(\widetilde{\rmbf C})= ||\rmbf X(\widehat{\rmbf C}-\widetilde{\rmbf C})||^2 - 2 \langle \rmbf Y - \rmbf X \widetilde{\rmbf C}, \rmbf X (\widehat{\rmbf C}- \widetilde{\rmbf C}) \rangle + n \sum\limits_{j=s+1}^p \lambda_j ||\widehat{\rmbf C}_j||.$$
	We have 
	\begin{align*}
	   | \langle \rmbf Y - \rmbf X \widetilde{\rmbf C}, \rmbf X (\widehat{\rmbf C}- \widetilde{\rmbf C}) \rangle | & =   |\langle \rmbf Y - \rmbf X\p1 \widehat{\rmbf C}\p1, \rmbf X\p2 \widehat{\rmbf C}\p2 \rangle |  \\
	  & =    |\langle \rmbf X\p2^T (\rmbf Y - \rmbf X\p1 \widehat{\rmbf C}\p1 ), \widehat{\rmbf C}\p2 \rangle |  \\
      & \leq    |\langle \rmbf X\p2^T \rmbf E, \widehat{\rmbf C}\p2 \rangle | +  |\langle \rmbf X\p2^T \rmbf X\p1( \widehat{\rmbf C}\p1-\rmbf C\p1^0 ), \widehat{\rmbf C}\p2 \rangle | \\
	  & \leq   \xi\p2 \sum\limits_{j= s+1}^p ||\widehat{\rmbf C}_j|| + \lambda_{\max}(\rmbf X\p2^T \rmbf X\p1) ||\widehat{\rmbf C}\p1 - \rmbf C^0\p1|| \, ||\widehat{\rmbf C}\p2||,
	\end{align*}
	where $\lambda_{\max}(\rmbf X\p2^T \rmbf X\p1)$ denotes the largest singular value of $\rmbf X\p2^T \rmbf X\p1$. Lemma \ref{lem:error order2} shows that $\xi\p2 \leq O_p(\sqrt{n q \log(p)})$. Thus
	\begin{align*}
	  Q(\widehat{\rmbf C})- Q(\widetilde{\rmbf C}) 
	  & \geq \sum\limits_{j= s+1}^p (n \lambda_j - 2 \, \xi\p2)  ||\widehat{\rmbf C}_j|| - 2 \, \lambda_{\max}(\rmbf X\p2^T \rmbf X\p1) ||\widehat{\rmbf C}\p1 - \rmbf C^0\p1|| \, ||\widehat{\rmbf C}\p2|| \\
	  & \geq  (n \lambda\p2 - 2 \, \xi\p2) \sum\limits_{j= s+1}^p ||\widehat{\rmbf C}_j|| - 2 \,  \lambda_{\max}(\rmbf X\p2^T \rmbf X\p1) ||\widehat{\rmbf C}\p1 - \rmbf C^0\p1|| \, ||\widehat{\rmbf C}\p2|| \\
	  & \geq (n \lambda\p2 - 2 \, \xi\p2)  ||\widehat{\rmbf C}\p2|| - 2 \, \lambda_{\max}(\rmbf X\p2^T \rmbf X\p1) ||\widehat{\rmbf C}\p1 - \rmbf C^0\p1|| \, ||\widehat{\rmbf C}\p2|| \\
	  & \geq  ||\widehat{\rmbf C}\p2|| \{ (n \lambda\p2 - 2 \, \xi\p2)  - 2 \, \lambda_{\max}(\rmbf X\p2^T \rmbf X\p1) ||\widehat{\rmbf C}\p1 - \rmbf C^0\p1|| \}, 
	\end{align*}
	where the third inequality is because the second part of Condition \eqref{con:variate lasso lambda} implies that $n \, \lambda\p2 \gg \xi\p2$ and $\sum_{j= s+1}^p ||\widehat{\rmbf C}_j|| \geq ||\widehat{\rmbf C}\p2||$. On the other hand, with the second part of Condition \eqref{con:variate lasso lambda} that $\lambda\p2 \gg \sqrt{r(q+s-r)/n}$, $||\widehat{\rmbf C}\p1 - \rmbf C^0\p1|| \leq O_p(\sqrt{r(q+s-r)/n})$, and $\lambda_{\max}(\rmbf X\p2^T \rmbf X\p1)=O(n)$, we have $n \, \lambda\p2 \gg \lambda_{\max}(\rmbf X\p2^T \rmbf X\p1) ||\widehat{\rmbf C}\p1 - \rmbf C^0\p1||$. Therefore,
	$$
	||\widehat{\rmbf C}\p2||\{ (n \lambda\p2 - 2 \, \xi\p2)  - 2 \, \lambda_{\max}(\rmbf X\p2^T \rmbf X\p1) ||\widehat{\rmbf C}\p1 - \rmbf C^0\p1|| \} > 0
	$$
	with conditional probability approaching $1$ given the event of $||\widehat{\rmbf C}\p2|| > 0$. In other words, on the event of $||\widehat{\rmbf C}\p2|| > 0$, the objective function at $\widetilde{\rmbf C}$ is smaller than that at $\widehat{\rmbf C}$, contradicting with the assumption that $\widehat{\rmbf C}$ is the minimizer of $Q(\rmbf C)$. This completes the second property stated in the theorem. 
\end{proof}

\begin{proof}[Proof of Lemma \ref{lem:group Lasso}]
	The proof of this Lemma is based on a modification of the proof of Theorem \ref{thm:convergence rate}. We similarly define $\delta_1 = \sum_{j =1}^s ||\widehat{\rmbf C}^{\mathrm{Lasso}}_j - \rmbf C^0_j ||$ and $\delta_2 = \sum_{j = s+1}^p ||\widehat{\rmbf C}^{\mathrm{Lasso}}_j||$.  When $\delta_2 > 2 \, \delta_1$, plugging the assumption $\lambda\p1 = \lambda\p2 = \lambda^{\mathrm{Lasso}} \gg \sqrt{q \log(p)/n}$ in \eqref{eqn:difference of Q1 case 2}, the same argument as in the proof of Theorem \ref{thm:convergence rate} shows that $\mathbb P(\delta_2 > 2 \, \delta_1) \to 0$ in the current setting. Given the event of $\delta_2 \leq 2\, \delta_1$, plugging the same assumption in \eqref{eqn:two point inequality}, \eqref{eqn:case 2 main} and \eqref{eqn:case 1 using lemma 1}, it is shown that
	\begin{align} \label{eqn:cor 1 main}
	   n \,  C\, \|\widehat{\rmbf C}^{\mathrm{Lasso}}\p1 - \rmbf C^0\p1 \|^2 
	   &\leq   \{O_p(\sqrt{nq  \log(p)})- n \, \lambda^{\mathrm{Lasso}} \} \delta_2 + \{ O_p (\sqrt{n q \log(s)} )  + n \,  \lambda^{\mathrm{Lasso}}\} \delta_1 \nonumber \\
	   &\leq  n \, \lambda^{\mathrm{Lasso}} \delta_1 \nonumber \\
	   &\leq  O\left(\frac{n^{1-\epsilon}}{s}\right) \delta_1 \nonumber \\
	   &\leq  O\left(\frac{n^{1-\epsilon}}{\sqrt{s}}\right) \|\widehat{\rmbf C}^{\mathrm{Lasso}}\p1 - \rmbf C^0\p1 \| 
	\end{align}
	with conditional probability approaching 1, where the second and third inequalities in \eqref{eqn:cor 1 main} is because of Condition \eqref{con:lasso lambda} on $\lambda^{\mathrm{Lasso}}$ and the forth inequality is a direct application of \eqref{eqn:CauchyOnDelta1}. Thus, $\|\widehat{\rmbf C}^{\mathrm{Lasso}}\p1 - \rmbf C\p1^0 \| \leq O_p(n^{-\epsilon}/\sqrt{s})$. Moreover, on the even of $\delta_2 \leq 2\,\delta_1$, $\|\widehat{\rmbf C}^{\mathrm{Lasso}}\p2\| \leq \delta_2 \leq 2 \, \delta_1 \leq 2 \sqrt{s} \, \|\widehat{\rmbf C}^{\mathrm{Lasso}}\p1 - \rmbf C\p1^0 \| =O_p(n^{-\epsilon})$ by \eqref{eqn:CauchyOnDelta1}. The combination of these two results implies that
	$$
	\|\widehat{\rmbf C}^{\mathrm{Lasso}} - \rmbf C^0 \| =  \|\widehat{\rmbf C}^{\mathrm{Lasso}}\p1 - \rmbf C\p1^0 \| + \|\widehat{\rmbf C}^{\mathrm{Lasso}}\p2\| \leq O_p(n^{-\epsilon}).
	$$ Finally, with $\epsilon >0$, Condition \eqref{con:uniform lower bound for all C} and the upper bound of the convergence rate of $\widehat{\rmbf C}^{\mathrm{Lasso}}$ imply that $$\|\widehat{\rmbf C}^{\mathrm{Lasso}}_j \| \geq \|\rmbf C^0_j \| - \|\widehat{\rmbf C}^{\mathrm{Lasso}}_j - \rmbf C^0_j \| > \|\rmbf C^0_j \|/2 \geq C/2 > 0 , \quad \forall \, j= 1, \dots, s$$
	for some constant $C$, with probability approaching $1$.
\end{proof}

\begin{proof}[Proof of Theorem \ref{thm:adaptive group Lasso}]
	We need to check that Condition \eqref{con:variate lasso lambda} used in Theorem \ref{thm:convergence rate} is satisfied. According to Lemma \ref{lem:group Lasso}, $\|\widehat{\rmbf C}^{\mathrm{Lasso}}_j \| > C/2$ with some constant $ C>0$, $j =1, \dots, s$. This result, together with the upper bound of $\lambda^{\mathrm{Adap}}$ stated in Condition \eqref{con:condition on beta}, imply that $\lambda^{\mathrm{Adap}}\, ||\widehat{\rmbf C}^{\mathrm{Lasso}}_j||^{-\beta} \leq O_p(\sqrt{r(q+s-r)/ns})$, $j =1, \dots, s$, thus the first part of Condition \eqref{con:variate lasso lambda} holds. It also follows from Lemma \ref{lem:group Lasso} that $||\widehat{\rmbf C}^{\mathrm{Lasso}}_j|| \leq O_p(n^{-\epsilon})$, $j = s+1, \dots, p$. This together with Condition \eqref{con:condition on beta} imply the second part of Condition \eqref{con:variate lasso lambda}. Finally, Condition \eqref{con:uniform lower bound for all C} and the upper bound of the convergence rate of $\widehat{\rmbf C}^{\mathrm{Adap}}$ stated in Theorem \ref{thm:convergence rate} imply that
	$$\|\widehat{\rmbf C}^{\mathrm{Adap}}_j \| \geq \|\rmbf C^0_j \| - \|\widehat{\rmbf C}^{\mathrm{Adap}}_j - \rmbf C^0_j \| > \|\rmbf C^0_j \|/2 \geq C/2 > 0 , \quad \forall \, j= 1, \dots, s$$
	for some constant $C$, with probability approaching $1$.
\end{proof}

\section*{Acknowledgement}
Huang's work was partially supported by NSF grant DMS-1208952.

\spacingset{1.00} 
\bibliographystyle{jasa}
\bibliography{reference}
\end{document}